\documentclass[twoside,11pt,reqno]{amsart}
\usepackage{amsmath,amssymb,amscd,mathrsfs,amscd, todonotes}
\usepackage{graphics,verbatim}
\usepackage{todonotes}
\usepackage{enumitem}
\usepackage{hyperref}

\newcommand{\losemi}{{\otimes \kern -.78em \ltimes}}
\newcommand{\rosemi}{{\otimes \kern -.78em \rtimes}}

\makeatletter
\newcommand{\leqnomode}{\tagsleft@true}
\newcommand{\reqnomode}{\tagsleft@false}
\makeatother

\newtheorem{theorem}{Theorem}[subsection]

\makeatletter\let\c@fact\c@theorem\makeatother

\makeatletter\let\c@note\c@theorem\makeatother

\newtheorem{lemma}{Lemma}[subsection]
\makeatletter\let\c@lemma\c@theorem\makeatother

\makeatletter\let\c@lemma\c@theorem\makeatother

\makeatletter\let\c@alg\c@theorem\makeatother

\newtheorem{prop}{Proposition}[subsection]
\makeatletter\let\c@prop\c@theorem\makeatother

\makeatletter\let\c@conj\c@theorem\makeatother

\newtheorem{cor}{Corollary}[subsection]
\makeatletter\let\c@cor\c@theorem\makeatother

\newtheorem{defn}{Definition}[subsection]
\makeatletter\let\c@defn\c@theorem\makeatother

\theoremstyle{definition}

\newtheorem{remark}{Remark}[subsection]
\makeatletter\let\c@remark\c@theorem\makeatother

\makeatletter\let\c@example\c@theorem\makeatother
\numberwithin{equation}{subsection}

\usepackage[capitalise]{cleveref}
\crefname{theorem}{Theorem}{Theorems}
\crefname{fact}{Fact}{Facts}
\crefname{note}{Note}{Notes}
\crefname{lemma}{Lemma}{Lemmas}
\crefname{alg}{Algorithm}{Algorithms}
\crefname{remark}{Remark}{Remarks}
\crefname{example}{Example}{Examples}
\crefname{prop}{Proposition}{Propositions}
\crefname{conj}{Conjecture}{Conjectures}
\crefname{cor}{Corollary}{Corollaries}
\crefname{defn}{Definition}{Definitions}
\crefname{equation}{\!\!}{\!\!} 

\newcounter{listequation}

\begin{document}
\title{On $(p,r)$-Filtrations and Tilting Modules}

\author{Paul Sobaje}
\address{Department of Mathematics \\
          University of Georgia \\
          Athens, GA 30602}
\email{sobaje@uga.edu}
\date{\today}

\begin{abstract}
We study the relationship between Donkin's Tilting Module Conjecture and Donkin's Good $(p,r)$-Filtration Conjecture.  Our main result was motivated by a result of Kildetoft and Nakano showing that the Tilting Module Conjecture implies one direction of the Good $(p,r)$-Filtration Conjecture.  We observe that the converse nearly holds; in particular a weaker version of the Good $(p,r)$-Filtration Conjecture implies the Tilting Module Conjecture.  
\end{abstract}

\maketitle

\section{Introduction}

\subsection{}
Let $k$ be an algebraically closed field of characteristic $p>0$, and let $G$ be a simple and simply connected algebraic group over $k$.  Let $G_r$ denote the $r$-th Frobenius kernel of $G$.  Fix a maximal torus $T$, and a set of dominant weights $X(T)_+$.  Let $X_r(T) \subseteq X(T)_+$ be the set of $p^r$-restricted dominant weights.

In this article we investigate the relationship between Donkin's Tilting Module Conjecture and Donkin's Good $(p,r)$-Filtration Conjecture.  Recall that for each $\lambda \in X(T)_+$ there is a simple module $L(\lambda)$ having highest weight $\lambda$, as well as highest weight modules $\Delta(\lambda)$ and $\nabla(\lambda)$, known as a Weyl module and an induced module respectively.  A $G$-module is said to have a good filtration (resp. Weyl filtration) if it has a filtration with quotients that are induced modules (resp. Weyl modules).  A finite dimensional tilting module is a finite dimensional $G$-module $M$ such that both $M$ and $M^*$ have a good filtration, and there is a unique indecomposable tilting module $T(\lambda)$ having highest weight $\lambda$.  Finally, $\lambda$ can be uniquely written as $\lambda_0 + p^r\lambda_1$, where $\lambda_0 \in X_r(T)$.  We can define the modules $\nabla^{(p,r)}(\lambda) := L(\lambda_0) \otimes \nabla(\lambda_1)^{(r)}$, and $\Delta^{(p,r)}(\lambda) := L(\lambda_0) \otimes \Delta(\lambda_1)^{(r)}$.  A module has a good $(p,r)$-filtration (resp. $(p,r)$-Weyl filtration) if it has a filtration such that each quotient is isomorphic to some $\nabla^{(p,r)}(\lambda)$ (resp. $\Delta^{(p,r)}(\lambda)$).

The $r$-th Steinberg module $\textup{St}_r$, the simple $G$-module of highest weight $(p^r-1)\rho$, plays a prominent role in the representation theory of $G$.  In 1990 at MSRI, Donkin formulated several conjectures that, when true, shed some light on this.  Consider the following statements:

\vspace{0.1in}

\noindent\textbf{(S1):} For every $\lambda \in X_r(T)$, $T((p^r-1)\rho+\lambda)$ is indecomposable over $G_r$.

\vspace{0.04in}

\noindent\textbf{(S2):} If $\textup{St}_r \otimes M$ has a good filtration, then $M$ has a good $(p,r)$-filtration.

\vspace{0.04in}

\noindent\textbf{(S3):} If $M$ has a good $(p,r)$-filtration, then $\textup{St}_r \otimes M$ has a good filtration.

\vspace{0.04in}

\noindent\textbf{(S4):} $\textup{St}_r \otimes L(\lambda)$ is a tilting module for every $\lambda \in X_r(T)$.

\vspace{0.04in}

\noindent\textbf{(S5):} $\nabla(\lambda)$ has a good $(p,r)$-filtration for every $\lambda \in X(T)_+$.

\vspace{0.1in}

\textbf{(S1)} is Donkin's Tilting Module Conjecture, and is known to hold when $p \ge 2h-2$, where $h$ is the Coxeter number of $G$ \cite[II.11]{J}.  It also holds for all $p$ when $G = SL_3$ (see recent work by Donkin on this \cite{D3}).  Statements \textbf{(S2)} and \textbf{(S3)}, taken together, comprise Donkin's Good $(p,r)$-Filtration Conjecture.  Andersen \cite{A} showed that \textbf{(S3)} is equivalent to \textbf{(S4)}, and proved that both hold when $p \ge 2h-2$.  Kildetoft and Nakano \cite{KN} gave two alternate proofs that \textbf{(S4)} holds when $p \ge 2h-2$, the second of which came by showing that \textbf{(S1)} implies \textbf{(S4)}.  As a consequence, it immediately follows that \textbf{(S4)} also holds for $SL_3$.  They also verified this last fact directly, as well as calculating that \textbf{(S4)} holds in all characteristics less than $2h-2$ for $SL_4$, and $G$ of types $B_2$ and $G_2$ (with the exception of $p=7$ in the $G_2$ case).  \textbf{(S5)}, which is a special case of \textbf{(S2)}, was shown by Parshall and Scott \cite{PS} to hold if $p \ge 2h-2$ and if Lusztig's character formula holds for all restricted weights (we note that they actually proved the analogous statement for Weyl modules, which is equivalent to the statement above).

\subsection{}
In recent work (\cite{So1}, \cite{So2}) we have studied Donkin's Tilting Module Conjecture, and the related issue of trying to find any $G$-structure for the projective indecomposable $G_r$-modules (the ``Humphreys-Verma Conjecture'').  The work of Kildetoft-Nakano is therefore of much interest as it pertains to these problems.  Since it provides a necessary condition for Donkin's Tilting Module Conjecture to be true, it could potentially be used to find a counterexample should one exist.  The fact that Kildetoft and Nakano were able to verify \textbf{(S4)} in a number of low rank cases where the status of \textbf{(S1)} is not known also suggests that the former may be an easier condition to check.  Furthermore, a clear question raised by their work is whether or not the converse statement, that \textbf{(S4)} implies \textbf{(S1)}, is true.  If it is, then we immediately have new cases for which the Tilting Module Conjecture, and therefore the Humphreys-Verma Conjecture, both hold.

It is this last question that is the primary thrust of this paper, though we give several other results throughout that we believe will be helpful in studying these conjectures going forward.  In order to present our main result in this direction, we will give two more conditions related to the five statements above.  For each $\lambda \in X_r(T)$, set $\hat{\lambda} := 2(p^r-1)\rho + w_0\lambda$ (to be precise, this notation should also reference $r$, since $\lambda$ is also in $X_s(T)$ for any $s > r$, but we will assume that an $r$ has been fixed).

It is well known that $L(\lambda)$ is a $G$-submodule of $T(\hat{\lambda})$ having multiplicity one.  We now formulate the following:

\vspace{0.1in}

\noindent \textbf{(S6):} $\textup{St}_r \otimes (T(\hat{\lambda})/L(\lambda))$ is tilting for every $\lambda \in X_r(T)$.

\vspace{0.04in}

\noindent \textbf{(S7):} $\nabla(\hat{\lambda})$ has a good $(p,r)$-filtration for every $\lambda \in X_r(T)$.

\vspace{0.1in}

Note that \textbf{(S7)} is just a special case of \textbf{(S5)}.  As for \textbf{(S6)}, tensoring with the $r$-th Steinberg module gives a short exact sequence
\begin{equation}\label{split}0 \rightarrow \textup{St}_r \otimes L(\lambda) \rightarrow \textup{St}_r \otimes T(\hat{\lambda}) \rightarrow \textup{St}_r \otimes (T(\hat{\lambda})/L(\lambda)) \rightarrow 0.\end{equation}
If $\textup{St}_r \otimes (T(\hat{\lambda})/L(\lambda))$ is tilting, then $\textup{St}_r \otimes L(\lambda)$ has a Weyl filtration, which then implies it is tilting since it is $\tau$-invariant (with $\tau$ as defined in the next section).  By basic properties of tilting modules, the sequence in (\ref{split}) would then split, and it is not hard to see that \textbf{(S6)} is in fact equivalent to the splitting of (\ref{split}).

We summarize our main results, which will be proved in Section 5.

\begin{theorem}
The following hold:

\vspace{0.08in}
\noindent \textup{(a)} Statements \textup{\textbf{(S1)}} and \textbf{\textup{(S6)}} are equivalent.

\vspace{0.05in}
\noindent \textup{(b)} Statements \textup{\textbf{(S4)}} and \textup{\textbf{(S7)}} together imply \textup{\textbf{(S1)}}.
\end{theorem}

Since \textup{\textbf{(S4)}} is equivalent to \textup{\textbf{(S3)}}, and \textup{\textbf{(S7)}} is a special case of \textup{\textbf{(S2)}}, we obtain as a corollary:

\begin{cor}
Donkin's Good $(p,r)$-Filtration Conjecture implies Donkin's Tilting Module Conjecture.
\end{cor}

\subsection{Acknowledgments} The author thanks Christopher Bendel, Stephen Donkin, Daniel Nakano, and Cornelius Pillen for helpful discussions and  correspondences.  This work was partially supported by the Research Training Grant, DMS-1344994, from the NSF.

\section{Preliminaries}

\subsection{}
All notation not introduced in this paper will follow the notation found in \cite{J}.  For every $\lambda \in X_r(T)$, set
$$\lambda^0 := (p^r-1)\rho + w_0\lambda.$$
Note that $\lambda^0 \in X_r(T)$ also, and comparing with earlier notation, 
$$\hat{\lambda} = (p^r-1)\rho + \lambda^0.$$

Let $Q_r(\lambda)$ be the $G_r$-projective cover of $L(\lambda)$.  It has a unique $W$-invariant lift to $G_rT$, denoted by $\widehat{Q}_r(\lambda)$, and this module is known to have highest weight $\hat{\lambda}$.

Fix a Borel subgroup $B$ containing $T$.  The negative roots will correspond to those root subspaces in $B$.  Denote by $\Pi$ the set of simple positive roots.

There is an antiautomorphism $\tau:G \rightarrow G$ that is the identity on $T$, and swaps the positive and negative root subgroups.  For a finite dimensional $G$-module $M$, we obtain the module ${^{\tau}M}$, which is $M^*$ as a vector space, with action $g.f(m)=f(\tau(g).m)$.  This defines a character-preserving anti-equivalence from the category of finite dimensional $G$-modules to itself, sending $M$ to $^{\tau}M$ (cf. \cite[II.2.12]{J}).  Simple modules and finite dimensional tilting modules are two classes of modules for which $M \cong {^{\tau}M}$, while $\tau$ takes modules with good filtrations to modules with Weyl filtrations, and vice versa.

\subsection{}

We will frequently use this next result, which is essentially \cite[Lemma E.9]{J}.

\begin{lemma}\label{filtrationtwists}
Let $\lambda \in X(T)_+$.  If $M$ has a good filtration (resp. Weyl filtration), then $T((p^r-1)\rho + \lambda) \otimes M^{(r)}$ has a good filtration (resp. Weyl filtration).
\end{lemma}

\begin{proof}
Suppose that $M$ has a good filtration.  We have $T((p^r-1)\rho + \lambda)$ as a summand of the tilting module $\textup{St}_r \otimes T(\lambda)$, therefore $T((p^r-1)\rho+\lambda) \otimes M^{(r)}$ is a summand of $\textup{St}_r \otimes T(\lambda) \otimes M^{(r)}$.  Since $\textup{St}_r \otimes M^{(r)}$ has a good filtration by \cite[II.3.19]{J}, the result follows.  The proof for Weyl filtrations is similar.
\end{proof}

\section{Variations on \textbf{(S3)} and \textbf{(S4)}}

A key equivalence established in \cite[Theorem 9.2.3]{KN} is that $\textup{St}_r \otimes M$ has a good filtration if and only if $\text{Hom}_{G_r}(T(\hat{\lambda}),M)^{(-r)}$ has a good filtration for every $\lambda \in X_r(T)$.  In this section we look at other ways to formulate these conditions, as well as providing a few preliminary consequences.

\subsection{}

For later use, we want to prove that $L(\mu)$, $\textup{Hom}_{G_r}(T(\hat{\lambda}),L(\mu))^{(-r)}$ has a good filtration for every $\lambda \in X_r(T)$ if and only if it is tilting.  We will do this by establishing the following general facts.

\begin{lemma}\label{untwisted}
For any finite dimensional $G$-module $M$, the following are equivalent.
\begin{enumerate}[label=(\alph*)]
\item $\textup{Hom}_{G_r}(T(\hat{\lambda}),M)^{(-r)}$ has a good filtration for every $\lambda \in X_r(T)$.
\item $\textup{Hom}_{G_r}(M^*,T(\hat{\lambda}))^{(-r)}$ has a good filtration for every $\lambda \in X_r(T)$.
\item $\textup{Hom}_{G_r}(X,M)^{(-r)}$ has a good filtration for every tilting module $X$ that is projective over $G_r$.
\item $\textup{Hom}_{G_r}(M^*,X)^{(-r)}$ has a good filtration for every tilting module $X$ that is projective over $G_r$.
\end{enumerate}
\end{lemma}

\begin{proof}
For any two finite dimensional $G$-modules $A,B$ there is an isomorphism of $G$-modules
$$\text{Hom}_{G_r}(A,B) \cong \text{Hom}_{G_r}(B^*,A^*).$$
Noting that $T(\hat{\lambda})^* \cong T(-w_0\hat{\lambda})$, and that if $X$ is tilting and projective over $G_r$, then $X^*$ is also, we see that (a) $\iff$ (b) and (c) $\iff$ (d).

It is clear that (d) implies (b), since every $T(\hat{\lambda})$ is projective over $G_r$.  Conversely, suppose that (b) holds.  If $X$ is projective over $G_r$ and tilting, then it is isomorphic to a direct summand of a tilting module of the form
$$\bigoplus_{\gamma \in \Gamma} (T(\hat{\gamma_0}) \otimes T(\gamma_1)^{(r)}),$$
where $\Gamma \subseteq X(T)_+$, and $\gamma = \gamma_0 + p^r\gamma_1$ with $\gamma_0 \in X_r(T)$ and $\gamma_1 \in X(T)_+$.

We then have that $\text{Hom}_{G_r}(M^*,X)^{(-r)}$ is a summand of
$$\textup{Hom}_{G_r}\left(M^*,\bigoplus_{\gamma \in \Gamma} (T(\hat{\gamma_0}) \otimes T(\gamma_1)^{(r)})\right)^{(-r)} \cong \bigoplus_{\gamma \in \Gamma} \left( \text{Hom}_{G_r}(M^*,T(\hat{\gamma_0}))^{(-r)} \otimes T(\gamma_1) \right).$$
Since this module has a good filtration, we conclude that $\text{Hom}_{G_r}(M^*,X)^{(-r)}$ does also.  Consequently, (b) $\Rightarrow$ (d).
\end{proof}

\begin{lemma}
Keep the assumptions on $M$ as in Lemma \ref{untwisted}.  If ${^{\tau}M} \cong M$, then $\textup{Hom}_{G_r}(X,M)^{(-r)}$ and $\textup{Hom}_{G_r}(M,X)^{(-r)}$ are tilting modules for every tilting module $X$ that is projective over $G_r$.
\end{lemma}

\begin{proof}
We will give the proof for $\text{Hom}_{G_r}(T(\hat{\lambda}),M)^{(-r)}$, from which the result can easily be generalized by similar arguments to those used above.  Suppose that any (hence all) of the equivalent conditions in Lemma \ref{untwisted} are satisfied.  Because $T(\hat{\lambda})$ is a $G$-summand of $\textup{St}_r \otimes T(\lambda^0)$, we have that $\text{Hom}_{G_r}(T(\hat{\lambda}),M)^{(-r)}$ is a $G$-summand of $\text{Hom}_{G_r}(\textup{St}_r \otimes T(\lambda^0),M)^{(-r)}$, which also has a good filtration.  There is a $G$-isomorphism
$$\text{Hom}_{G_r}(\textup{St}_r \otimes T(\lambda^0),M) \cong \text{Hom}_{G_r}(\textup{St}_r, T(\lambda^0)^* \otimes M).$$
Further, $\textup{St}_r \otimes \text{Hom}_{G_r}(\textup{St}_r, T(\lambda^0)^* \otimes M)$ is a $G$-summand of $T(\lambda^0)^* \otimes M$ \cite[II.10.4(a)]{J}.  Since $T(\lambda^0)^* \otimes M$ is $\tau$-invariant, this summand must be also (since $\textup{St}_r$ is $\tau$-invariant over both $G$ and $G_r$), hence $\text{Hom}_{G_r}(\textup{St}_r, T(\lambda^0)^* \otimes M)$ and $\text{Hom}_{G_r}(\textup{St}_r \otimes T(\lambda^0),M)$ are $\tau$-invariant.  It follows that $\text{Hom}_{G_r}(\textup{St}_r \otimes T(\lambda^0),M)^{(-r)}$ is a tilting module, thus that its summand $\text{Hom}_{G_r}(T(\hat{\lambda}),M)^{(-r)}$ is tilting.
\end{proof}

\begin{cor}
Statement \textbf{\textup{(S4)}} holds if and only if $\textup{Hom}_{G_r}(\textup{St}_r,L(\lambda) \otimes L(\mu))^{(-r)}$ is a tilting module for every $\lambda,\mu \in X_r(T)$.
\end{cor}

\begin{proof}
If \textbf{(S4)} holds, then each $\textup{St}_r \otimes L(\mu)$ is tilting, and hence by the previous two results
$$\textup{Hom}_{G_r}(\textup{St}_r,L(\lambda) \otimes L(\mu))^{(-r)} \cong \textup{Hom}_{G_r}(\textup{St}_r \otimes L(\mu)^*, \otimes L(\lambda))^{(-r)}$$
is also tilting.

Conversely, suppose that $\textup{Hom}_{G_r}(\textup{St}_r,L(\lambda) \otimes L(\mu))^{(-r)}$ is tilting for each $\lambda,\mu \in X_r(T)$.  Each $T(\hat{\gamma})$ appears as a $G$-summand of a module of the form $\textup{St}_r \otimes L(\mu)^*$ (specifically, for $\mu = (p^r-1)\rho - \gamma$), so we have that
$$\textup{Hom}_{G_r}(T(\hat{\gamma}), L(\lambda))^{(-r)}$$
will be tilting as $\lambda,\gamma$ range over all pairs of elements in $X_r(T)$, hence $\textbf{(S4)}$ holds.
\end{proof}

\begin{remark}
A necessary condition for $\textup{St}_r \otimes L(\mu)$ to be tilting is that $\textup{Hom}_{G_r}(\textup{St}_r \otimes \textup{St}_r, \otimes L(\mu))^{(-r)}$ is tilting.  There is an isomorphism of $T$-modules
$$\textup{Hom}_{G_r}(\textup{St}_r \otimes \textup{St}_r, \otimes L(\mu)) \cong \textup{Hom}_{T_r}(k, \otimes L(\mu)) \cong L(\mu)^{T_r}.$$
So if \textbf{(S4)} holds, then $\text{ch}(L(\mu)^{T_r})$ is $p^r$-times the character of a tilting module for every $\mu \in X_r(T)$.
\end{remark}

\subsection{}

Andersen proved that if $M$ and $N$ are $G$-modules such that both $\textup{St}_r \otimes M$ and $\textup{St}_r \otimes N$ have a good filtration, then $\textup{St}_r \otimes M \otimes N$ has a good filtration \cite[Proposition 4.4]{A}.  We use similar arguments in establishing the next proposition.

\begin{prop}
Let $M$ be a finite dimensional $G$-module.
\begin{enumerate}[label=(\alph*)]
\item $\textup{St}_r \otimes M$ has a good filtration if and only if ${\textup{St}_r}^{\otimes n} \otimes M$ has a good filtration for some $n \ge 1$.
\item If $\textup{St}_r \otimes M \otimes L(\lambda)$ has a good filtration for some $\lambda \in X_r(T)$, then $\textup{St}_r \otimes M$ has a good filtration.
\end{enumerate}
\end{prop}

\begin{proof}
(a) If ${\textup{St}_r}^{\otimes n} \otimes M$ has a good filtration, then ${\textup{St}_r}^{\otimes n+i} \otimes M$ does also for all $i \ge 1$.  Since $\text{St}_r$ is self-dual, it follows that $\textup{St}_r$ is a $G$-summand of $\textup{St}_r \otimes \textup{St}_r \otimes \textup{St}_r$ (see, for example, the end of the proof of Theorem 2.1 in \cite{CPS}).  From this last fact the result is easily deduced.

(b) Suppose that $\textup{St}_r \otimes M \otimes L(\lambda)$ has a good filtration, with $\lambda \in X_r(T)$.  Then
$$V = \textup{St}_r \otimes M \otimes L(\lambda) \otimes T((p^r-1)\rho - \lambda)$$
has a good filtration also.  But $\textup{St}_r$ is a summand of $L(\lambda) \otimes T((p^r-1)\rho - \lambda)$, so ${\textup{St}_r}^{\otimes 2} \otimes M$ is a summand of $V$ and therefore has a good filtration.  By (a) the result follows.

\end{proof}

\begin{remark}
Though $X_r(T)$ is not a minimal set of weights in $X(T)_+$ under the usual ordering (that is, there is generally some $\sigma < \lambda$ with $\lambda \in X_r(T)$ and $\sigma \not\in X_r(T)$), they are precisely the dominant weights $\lambda$ for which $(p^r-1)\rho - \lambda$ is also dominant.  We see the importance of this last fact highlighted in the ``cancellation'' property of $L(\lambda)$ in part (b) of the proposition.
\end{remark}

\section{Decomposing Tilting Modules}

In this section we look at how various tilting modules decompose over $G$ and $G_r$.  Since $\text{St}_r \otimes L(\lambda)$, where $\lambda \in X_r(T)$, is known to be tilting in many cases, these results immediately apply to such modules (see also \cite{K} for decompositions of modules of the form $\text{St}_r \otimes \nabla(\lambda)$).

\subsection{}

We begin by recalling the ``rational order'' on $X(T)$.

\begin{defn}
The order relation $\le_{\mathbb{Q}}$ on $X(T)$ is given by $\lambda \le_{\mathbb{Q}} \mu$ if
$$\mu - \lambda = \sum_{\alpha \in \Pi} q_{\alpha}\alpha, \quad q_{\alpha} \in \mathbb{Q}_{\ge 0}.$$
\end{defn}

It is clear that if $\lambda \le \mu$, then $\lambda \le_{\mathbb{Q}} \mu$.

\begin{lemma}
If $\lambda \in X(T)_+$, then $\lambda \ge_{\mathbb{Q}} 0$.
\end{lemma}

\begin{proof}
This can be found, for example, in \cite[13.2]{H}.
\end{proof}

This order can now be used to formulate an important (and to our knowledge previously unobserved) fact about $G_r$-decompositions of modules of the form $\textup{St}_r \otimes L(\lambda^0)$ when $\lambda \in X_r(T)$ (recall that $\lambda^0=(p^r-1)\rho+w_0\lambda$).  Namely, that if $Q_r(\mu)$ is a $G_r$-summand of $\textup{St}_r \otimes L(\lambda^0)$, then $\mu \ge_{\mathbb{Q}} \lambda$.

\begin{prop}\label{highestweightmodule}
Let $\lambda \in X_r(T)$, and let $P$ be a finite dimensional $G$-module such that:
\begin{enumerate}[label=(\alph*)]
\item $P$ is projective over $G_r$.
\item $\hat{\lambda}$ is the highest weight of $P$.
\item $P_{\hat{\lambda}}$ is $1$-dimensional.
\end{enumerate}
Then as a $G_rT$-module,
$$P \cong \widehat{Q}_r(\lambda) \oplus \bigoplus_{\mu \in X_r(T), \, \mu >_{\mathbb{Q}} \lambda} \left(\widehat{Q}_r(\mu) \otimes \textup{Hom}_{G_r}(L(\mu),P)\right).$$
In any such decomposition, the $G_r$-socle of $\widehat{Q}_r(\lambda)$ is the unique $G$-submodule of $P$ isomorphic to $L(\lambda)$.
\end{prop}

\begin{proof}
First, the projectivity of $P$ over $G_rT$ means that it decomposes into $G_rT$-projective indecomposable modules.  These summands are determined completely by the $G_rT$-socle of $P$, which coincides with the $G_r$-socle.  As a $G$-module, hence as a $G_rT$-module, the $G_r$-socle is isomorphic to
$$\bigoplus_{\mu \in X_r(T)} \left(L(\mu) \otimes \textup{Hom}_{G_r}(L(\mu),P) \right).$$
From this we have that as $G_rT$-modules,
$$P \cong \bigoplus_{\mu \in X_r(T)} \left(\widehat{Q}_r(\mu) \otimes \textup{Hom}_{G_r}(L(\mu),P)\right).$$
The highest weights appearing on the right-hand side of this isomorphism have the form
$$\hat{\mu} + p^r\gamma, \quad \mu \in X_r(T), \gamma \in X(T)_+.$$
Since $\hat{\lambda}$ is the highest weight of $P$, and $P_{\hat{\lambda}}$ is one-dimensional, it follows that $\widehat{Q}_r(\lambda)$ must occur with multiplicity one.  Further, we see that if $\mu \in X_r(T)$ appears in the decomposition, then for some $\gamma \in X(T)_+$ we have that $\hat{\lambda} > \hat{\mu} + p^r\gamma$.  Subtracting $2(p^r-1)\rho$ from each side of the inequality, we have $w_0\lambda > w_0\mu + p^r\gamma$.  Since $p^r\gamma \ge_{\mathbb{Q}} 0$, it follows that
\begin{align*}
& w_0\lambda > w_0\mu + p^r\gamma \ge_{\mathbb{Q}} w_0\mu + 0\\
\Rightarrow \qquad & w_0\lambda >_{\mathbb{Q}} w_0\mu\\
\Rightarrow \qquad & \mu >_{\mathbb{Q}} \lambda.
\end{align*}

Finally, the isotypic components of the $G_r$-socle of $P$ are $G$-submodules of $P$, so $L(\lambda)$ must occur as a $G$-submodule exactly once. 
\end{proof}

\begin{remark}
In particular, the conditions in Proposition \ref{highestweightmodule} are satisfied if $P$ is isomorphic to any of the following modules:
$$\{\textup{St}_r \otimes L(\lambda^0), \, \textup{St}_r \otimes \nabla(\lambda^0), \, \textup{St}_r \otimes \Delta(\lambda^0), \, \textup{St}_r \otimes T(\lambda^0), \, T(\hat{\lambda})\}.$$
Basic properties of tilting modules show that $T(\hat{\lambda})$ is a $G$-summand of $\textup{St}_r \otimes T(\lambda^0)$, and it turns out to be a $G$-summand of every module in the set above, thanks to a result of Pillen \cite[Corollary A]{P}.
\end{remark}

\begin{theorem}
The following hold:

\begin{enumerate}[label=(\alph*)]
\item If $\lambda \in X_r(T)$ is maximal under $\le_{\mathbb{Q}}$ among the weights in its $G_r$-block, then $T(\hat{\lambda})$ is indecomposable over $G_r$.
\item If $\lambda \in X_r(T)$ is minimal in $X(T)_+$ under $\le$, then $\textup{St}_r \otimes L(\lambda)$ is indecomposable over $G_r$.  Consequently, $\textup{St}_r \otimes L(\lambda) \cong T((p^r-1)\rho + \lambda)$, and is a $G$-structure for $ Q_r(\lambda^0)$.
\end{enumerate}

\end{theorem}

\begin{proof}
(a) follows from Proposition \ref{highestweightmodule}, and the fact that any $G$-module decomposes as a direct sum according to the blocks of $G_r$.  For (b), if $Q_r(\mu)$ is a summand of $\text{St}_r \otimes L(\lambda)$, then modifying the proof of Proposition \ref{highestweightmodule} (to account for the top weight being $(p^r-1)\rho + \lambda$ rather than $\hat{\lambda}=(p^r-1)\rho + \lambda^0$) we can deduce that there is some $\gamma \in X(T)_+$ such that $(p^r-1)\rho + \lambda \ge \hat{\mu} + p^r\gamma$.  This implies that $\lambda \ge (p^r-1)\rho + w_0\mu + p^r\gamma$.

Since $(p^r-1)\rho + w_0\mu + p\gamma \in X(T)_+$, and $\lambda$ is minimal under $\le$, it follows that $\gamma =0$ and $\lambda = (p^r-1)\rho + w_0\mu$, so that $\mu = (p^r-1)\rho + w_0\lambda$.  Thus, over $G_r$ we have
$$\text{St}_r \otimes L(\lambda) \cong Q_r(\lambda^0),$$
which implies that as $G$-modules
$$\text{St}_r \otimes L(\lambda) \cong T((p^r-1)\rho + \lambda).$$
\end{proof}

This recovers the following observation by Doty, which he used to obtain some interesting factorization results on tilting modules (cf. \cite{Do}).  We note that our proof follows immediately from the fact that the minuscule weights are precisely the minimal dominant weights under $\le$, and does not rely on Brauer's formula and Donkin's character computation (cf. Theorem 5.4 and Proposition 5.5 of \cite{D2}).
 
\begin{cor} \textup{(Doty \cite{Do})}
If $\lambda$ is a minuscule weight, then $\textup{St}_r \otimes L(\lambda) \cong T((p^r-1)\rho + \lambda)$.
\end{cor}

\subsection{}

Following the ideas above, one way to show that $T(\hat{\lambda})$ is indecomposable over $G_r$ (if it actually is) is to prove that $Q_r(\mu)$ is not a $G_r$-summand of $T(\hat{\lambda})$ whenever $\mu \ne \lambda$.  We already know that we can restrict our consideration to just those $\mu >_{\mathbb{Q}} \lambda$.  Further, we may assume that $T(\hat{\mu})$ is indecomposable over $G_r$ (for if we could prove the statement here, the rest would follow by induction).

The following lemmas will help us analyze this situation.

\begin{lemma}\label{G_rsplit}
Suppose that $X \le Y$ are tilting modules such that $X$ is projective over $G_r$.  Suppose further that the $G_r$-summands appearing in a $G_r$-complement to $X$ in $Y$ are different (up to isomorphism) from the $G_r$-summands appearing in $X$.  Then $X$ is a $G$-summand of $Y$.  
\end{lemma}

\begin{proof}
Applying $\tau$ to the inclusion of $i:X \rightarrow Y$, we get a $G$-module homomorphism
$${^{\tau}{i}}: {^{\tau}Y} \rightarrow {^{\tau}X}.$$
Since $X,Y$ are tilting modules, ${^{\tau}Y} \cong Y$ and ${^{\tau}X} \cong X$.  This means that there is a $G$-submodule $M \le Y$ such that $Y/M \cong X$.  We now need to show that $M$ is a vector space complement to the submodule $X$.

By the Krull-Schmidt theorem, and our assumption above, the $G_r$-summands of $M$ must have distinct isomorphism-types from the $G_r$-summands in $X$.  Suppose now that $X \cap M \ne \{0\}$.  Then there is a simple $G_r$-submodule $L(\lambda)$ in this intersection.  Since $X$ is an injective $G_r$-module, the inclusion $L(\lambda) \rightarrow X$ extends to an injective $G_r$-homomorphism $Q_r(\lambda) \rightarrow X$.  Now take the projection over $G_r$,
$$\text{pr}_M: Y \rightarrow M.$$
It follows that the summand $Q_r(\lambda)$ in $X$ injects into $M$ via this projection.  Thus $Q_r(\lambda)$ is a $G_r$-summand of $M$.  This contradicts our assumption on the $G_r$-summands of $M$, forcing $X \cap M = \{0\}$.  It now follows that $Y = X + M$. 
\end{proof}

\begin{lemma}\label{assumeindecomposable}
Let $\lambda, \mu \in X_r(T)$, and let $M = \textup{Hom}_{G_r}(L(\mu),T(\hat{\lambda}))$.  Suppose that $T(\hat{\mu})$ is indecomposable over $G_r$, and that
$$\textup{Ext}^1_G((T(\hat{\mu})/L(\mu)) \otimes M, T(\hat{\lambda})) = 0.$$
Then $Q_r(\mu)$ is not a $G_r$-summand of $T(\hat{\lambda})$.
\end{lemma}

\begin{proof}
Apply the functor $\text{Hom}_G(\, \underline{\hspace{0.1in}}\,,T(\hat{\lambda}))$ to the inclusion of $G$-modules
$$L(\mu) \otimes M \hookrightarrow T(\hat{\mu}) \otimes M.$$
Since $\textup{Ext}^1_G((T(\hat{\mu})/L(\mu)) \otimes M, T(\hat{\lambda})) = 0$, we get a short exact sequence
$$0 \rightarrow \text{Hom}_G((T(\hat{\mu})/L(\mu) \otimes M, T(\hat{\lambda})) \rightarrow \text{Hom}_G(T(\hat{\mu}) \otimes M, T(\hat{\lambda}))$$
$$\qquad \xrightarrow{\psi} \text{Hom}_G(L(\mu) \otimes M, T(\hat{\lambda})) \rightarrow 0.$$
Since $T(\hat{\mu})$ is indecomposable over $G_r$, we have that $L(\mu) \otimes M$ contains the $G_r$-socle of $T(\hat{\mu}) \otimes M$, so it must also contain the $G$-socle of this module.  This fact, together with the surjectivity of $\psi$, shows that there is an injective $G$-homomorphism
$$T(\hat{\mu}) \otimes M \rightarrow T(\hat{\lambda}).$$
By Lemma \ref{G_rsplit}, this inclusion splits over $G$.  Since $T(\hat{\lambda})$ is indecomposable over $G$, it follows that $M=\{0\}$, finishing the proof.
\end{proof}

We conclude with a lemma that will be used in the next section.

\begin{lemma}\label{homsplits}
Let $N\le M$ be $G$-modules.  If $M$ is projective over $G_r$, and $\textup{Hom}_{G_r}(M,M/N)$ is one-dimensional, then there is a $G$-module decomposition
$$\textup{End}_{G_r}(M) \cong \textup{Hom}_{G_r}(M,N) \oplus \textup{Hom}_{G_r}(M,M/N).$$
\end{lemma}

\begin{proof}
Because of the projectivity of $M$ over $G_r$, we get a short exact sequence
$$0 \rightarrow \text{Hom}_{G_r}(M,N) \rightarrow \text{Hom}_{G_r}(M,M) \rightarrow \text{Hom}_{G_r}(M,M/N) \rightarrow 0,$$
and the $k$-span of the identity map in $\text{Hom}_{G_r}(M,M)$ clearly is a $G$-submodule that maps isomorphically onto $\text{Hom}_{G_r}(M,M/N)$, thus defining a splitting.
\end{proof}

\section{Main Results}

\subsection{}

\begin{theorem}
The following statements are equivalent:
\begin{enumerate}[label=(\alph*)]
\item $T(\hat{\lambda})$ is indecomposable over $G_r$ for every $\lambda \in X_r(T)$.
\item $\textup{St}_r \otimes (T(\hat{\lambda})/L(\lambda))$ is tilting for every $\lambda \in X_r(T)$.
\end{enumerate}
\end{theorem}

\begin{proof}
($(a) \Rightarrow (b)$): Let $\mu,\lambda \in X_r(T)$.  By \cite[Theorem 9.2.3]{KN}, $\text{Hom}_{G_r}(T(\hat{\mu}),T(\hat{\lambda}))^{(-r)}$ has a good filtration.  If each $T(\hat{\lambda})$ is indecomposable over $G_r$, then when $\lambda \ne \mu$ we have
$$\text{Hom}_{G_r}(T(\hat{\mu}),T(\hat{\lambda}))^{(-r)} \cong \text{Hom}_{G_r}(T(\hat{\mu}),\text{rad}_{G_r} \, T(\hat{\lambda}))^{(-r)}.$$
Thus, in these cases $\text{Hom}_{G_r}(T(\hat{\mu}),\text{rad}_{G_r} \, T(\hat{\lambda}))^{(-r)}$ has a good filtration.  Also, by Lemma \ref{homsplits}, $\text{Hom}_{G_r}(T(\hat{\lambda}),\text{rad}_{G_r} \, T(\hat{\lambda}))^{(-r)}$ has a good filtration since it is a $G$-summand of a module with a good filtration.  Using \cite[Theorem 9.2.3]{KN} again, we find that $\textup{St}_r \otimes \text{rad}_{G_r} \, T(\hat{\lambda})$ has a good filtration.  By applying $\tau$, it then follows that $\textup{St}_r \otimes (T(\hat{\lambda})/L(\lambda))$ has a Weyl filtration.

On the other hand, since $\textup{St}_r \otimes T(\hat{\lambda})$ and $\textup{St}_r \otimes L(\lambda)$ have good filtrations, we also see that $\textup{St}_r \otimes (T(\hat{\lambda})/L(\lambda))$ has a good filtration, hence it is tilting.

($(b) \Rightarrow (a)$): Suppose that $\textup{St}_r \otimes (T(\hat{\lambda})/L(\lambda))$ is tilting for every $\lambda \in X_r(T)$.  As noted in the introduction, this implies that $\text{St}_r \otimes L(\lambda)$ is also tilting.  By Lemma \ref{untwisted} we then have that $\textup{Hom}_{G_r}(L(\mu),T(\hat{\lambda}))^{(-r)}$ is tilting for all $\lambda,\mu \in X_r(T)$.  Additionally, using the fact that $T(\hat{\lambda})^*$ is the summand of $\text{St}_r \otimes T(\lambda^0)^*$, it is not hard to see that for every $\lambda,\mu \in X_r(T)$, we have that $T(\hat{\lambda})^* \otimes T(\hat{\mu})/L(\mu)$ is tilting.  Setting $M=\textup{Hom}_{G_r}(L(\mu),T(\hat{\lambda}))$, an application of Lemma \ref{filtrationtwists} shows that $T(\hat{\lambda})^* \otimes T(\hat{\mu})/L(\mu) \otimes M$ is tilting.  Therefore,
$$\textup{Ext}^1_G((T(\hat{\mu})/L(\mu)) \otimes M, T(\hat{\lambda})) \cong \textup{Ext}^1_G(T(\hat{\lambda})^* \otimes (T(\hat{\mu})/L(\mu)) \otimes M, k) = 0.$$
Suppose now that $\lambda \in X_r(T)$ is such that $T(\hat{\mu})$ is indecomposable over $G_r$ for all $\mu \in X_r(T)$ with $\mu >_{\mathbb{Q}} \lambda$.  By the preceding arguments, we may apply Lemma \ref{assumeindecomposable} to conclude that $Q_r(\mu)$ is not a $G_r$-summand of $T(\hat{\lambda})$, therefore $T(\hat{\lambda})$ is indecomposable over $G_r$ by Proposition \ref{highestweightmodule}.  The proof for all $\lambda \in X_r(T)$ now follows by induction.
\end{proof}

\subsection{}

\begin{theorem}
Suppose that for every $\lambda \in X_r(T)$ the following hold:
\begin{enumerate}[label=(\alph*)]
\item $\textup{St}_r \otimes L(\lambda)$ is tilting.
\item $\nabla(\hat{\lambda})$ has a good $(p,r)$-filtration.
\end{enumerate}
Then every $T(\hat{\lambda})$ is indecomposable over $G_r$.
\end{theorem}

\begin{proof}
Let $\lambda \in X_r(T)$, and suppose that for every $\mu \in X_r(T)$ such that $\mu >_{\mathbb{Q}} \lambda$, $T(\hat{\mu})$ is indecomposable over $G_r$.  Fix some such $\mu$.  Since $\nabla(\hat{\mu})$ is at the top of any good filtration on  $T(\hat{\mu})$, it also follows that $\nabla(\hat{\mu})$ is indecomposable over $G_r$ with simple $G_r$-head $L(\mu)$.  If $\nabla(\hat{\mu})$ has a good $(p,r)$-filtration, then $L(\mu)$ must occur as the final quotient of any such filtration.  This implies that $\text{rad}_{G_r} \, \nabla(\hat{\mu})$ has a good $(p,r)$-filtration.  Since condition $\textbf{(S4)}$ holds, $\textup{St}_r \otimes \text{rad}_{G_r} \, \nabla(\hat{\mu})$ has a good filtration.  By the same reasoning as in the proof of the previous theorem, this implies that $\textup{St}_r \otimes (T(\hat{\mu})/L(\mu))$ is tilting.  Applying Lemma \ref{assumeindecomposable}, we have that $Q_r(\mu)$ is not a $G_r$-summand of $T(\hat{\lambda})$.  The proof for all $\lambda \in X_r(T)$ now follows by induction.
\end{proof}

\begin{cor}
Donkin's Good $(p,r)$-Filtration Conjecture implies Donkin's Tilting Module Conjecture.
\end{cor}

\end{document}